\documentclass[a4paper,11pt]{article}

\usepackage[utf8]{inputenc}
\usepackage{amsmath}
\usepackage{amsfonts}
\usepackage{amssymb}
\usepackage{amsthm}
\usepackage{color}
\usepackage{esint}
\usepackage{graphicx}
\usepackage{comment}
\usepackage[a4paper,textwidth=14cm]{geometry}
\usepackage{enumitem}

\newcommand{\B}{\mathcal B}
\newcommand{\C}{\mathcal C}
\newcommand{\E}{\mathcal{E}}

\renewcommand{\H}{\mathcal{H}}
\newcommand{\N}{\mathbb{N}}
\newcommand{\R}{\mathbb{R}}
\renewcommand{\S}{\mathbb{S}}
\newcommand{\T}{\mathcal{T}}

\renewcommand{\div}{\mathrm{div}\,}
\newcommand{\diam}{\, \mathrm{diam}\,}

\newcommand{\eps}{\varepsilon}
\newtheorem{theorem}{Theorem}[section]
\newtheorem{lemma}[theorem]{Lemma}
\newtheorem{proposition}[theorem]{Proposition}
\newtheorem{corollary}[theorem]{Corollary}
\newtheorem{definition}[theorem]{Definition}

\theoremstyle{definition}
\newtheorem{remark}[theorem]{Remark}

\usepackage[hidelinks]{hyperref}

\title{Sticky-disk limit of planar $N$-bubbles}
\author{Giacomo Del Nin\thanks{Università di Pisa, \url{delnin@mail.dm.unipi.it}.}}
\date{}
\begin{document}
\maketitle

\begin{abstract}
We study planar $N$-bubbles that minimize, under an area constraint, a weighted perimeter $P_\eps$ depending on a small parameter $\eps>0$. Specifically we weight $2-\eps$ the boundary between the bubbles and $1$ the boundary between a bubble and the exterior. We prove that as $\eps\to 0$ minimizers of $P_\eps$ converge to configurations of disjoint disks that maximize the number of tangencies, each weighted by the harmonic mean of the radii of the two tangent disks. We also obtain some information on the structure of minimizers for small $\eps$.\\

{\footnotesize \emph{Keywords}: planar clusters, weighted perimeter, isoperimetric inequality, sticky~disk.

}
\end{abstract}

\section{Introduction}
In this work we are interested in studying the optimal way to enclose and separate $N$ areas in the plane in order to minimize a specific weighted perimeter.

An $N$-bubble, or $N$-cluster, is a family $\E=(\E(1),\ldots \E(N))$ of disjoint sets in the plane, called \emph{bubbles} or \emph{chambers}, that can touch only at their boundaries. The weighted perimeter of an $N$-bubble is given by the weighted sum of the lengths of all the interfaces, that is
\begin{equation}\label{eq:P}
P(\E)=\frac12 \sum_{\substack{0\leq i,j\leq N\\i\neq j}} c_{ij} \,\mathrm{length}\big( \partial \E(i)\cap \partial \E(j)\big)
\end{equation}
for some fixed positive weights $c_{ji}=c_{ij}>0$.  In the following we will fix the areas $m_1,\ldots,m_N$ of the bubbles and seek the configurations that minimize the perimeter $P(\E)$ under this constraint.

The exact characterization of perimeter minimizing $N$-bubbles is currently known only in very few situations. The case $N=1$ is the classical isoperimetric problem, whose well-known solution is a disk. If $N=2$ the solution is the \emph{standard weighted double bubble} made of three circular arcs meeting in two triple points forming angles which depend on the specific weights (see \cite{F+93} in the case of unit weights, \cite{Law12} in general). If $N=3$ the solution is known only for equal weights ($c_{ij}=1$), and it is the \emph{standard triple bubble} made of six circular arcs meeting in four points \cite{Wic04}. When $N=4$ and the weights are equal the minimal configuration has a determined topology and is conjectured to be the \emph{symmetric sandwich} \cite{PaoTam16}.

For general $N$ only existence and regularity of minimizers is known: under the strict triangle inequalities $c_{ij}< c_{ik}+c_{kj}$ for any distinct $i,j,k$, minimizers exist and their boundary is made of a finite number of circular arcs, meeting at a finite number of singular points where they satisfy a condition on the incidence angles \cite[Proposition~4.3]{Mor98}.

The exact characterization of minimizers seems an intractable problem already for small values of $N$. For this reason, in this work we consider a special asymptotic regime. Indeed for $\eps\geq 0$ we define
\begin{align}\label{eq:Peps}
\begin{split}
P_\eps(\E)&=\frac12 \sum_{\substack{0\leq i,j\leq N\\ i\neq j}}c_{ij}(\eps)\mathrm{length} (\partial \E(i)\cap \partial \E(j)),\\
c_{ij}(\eps)&=
\begin{cases}
1& \text{ if } i=0 \text{ or } j=0\\
2-\eps & \text{ if } i,j\neq 0
\end{cases}.
\end{split}
\end{align}

{\bf Problem} We want to study the asymptotic behaviour as $\eps \to 0$ of $N$-bubbles which minimize the energy $P_\eps$ with an area constraint $|\E(i)|=m_i$ for $i=1,\ldots, N$.\\

We denote by $\overline \E_\eps$ minimizers of $P_\eps$. We call a \emph{cluster of disks} any cluster made of disks with pairwise disjoint interiors.

\begin{proposition}[First-order behaviour]\label{prop:round}
As $\eps\to 0$ minimizers of $P_\eps$ converge to a cluster of disks.
\end{proposition}

\begin{figure}\label{fig:Asymptotic_N-bubbles}
\centering
\includegraphics[width=0.7\textwidth]{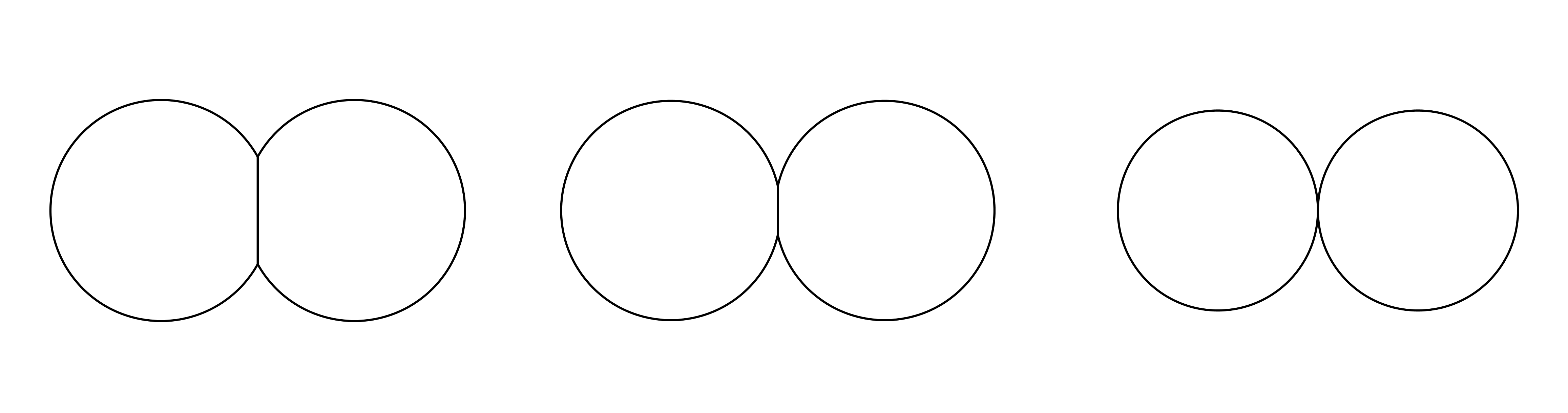}
\caption{\small{When $N=2$ we know the explicit shape of the (unique) minimizers of $P_\eps$, and as $\eps\to 0$ (from left o right) they converge to two tangent disks. Depicted here is the case of equal areas.}}
\end{figure}

At this level however we have no information on the disposition of the limit disks, since any collection of disks with pairwise disjoint interiors is a possible candidate. On the other hand we expect to see only certain configurations of disks: if we look for instance at the case $N=2$ with equal areas the limit disks must be tangent (see Figure \ref{fig:Asymptotic_N-bubbles}). To obtain more information we then perform a second-order expansion of the perimeter functional, that is we subtract the limit energy $P_0(\B)=\sum_{i=1}^N P(B_i)$, rescale by the right power of $\eps$ and analyze these rescaled functionals. To find the right scaling we look again at the completely solved case of two bubbles with equal areas $|\E(1)|=|\E(2)|=\pi$: an explicit computation shows that
\begin{equation}\label{eq:explicit}
P_\eps(\overline{\E_\eps})=4\pi -\frac43 \eps^{3/2} +O(\eps^{5/2})
\end{equation}
hence the relevant next order is $\eps^{3/2}$ and we are led to consider the rescaled functionals
\begin{equation}\label{eq:P1eps}
P^{(1)}_\eps(\E):=\frac{P_\eps(\E)-P_0(\B)}{\frac43\eps^{3/2}}.
\end{equation}
Of course they have the same minimizers as $P_\eps$ but allow us to analyze the finer behaviour at scale $\eps^{3/2}$. We expect that, as in the case of the double bubble, these functional “see” the tangency points in the limit cluster $\B$. Indeed, this is precisely what happens. The following is the main result of this work:

\begin{theorem}[Sticky-disk limit]\label{thm:sticky}
As $\eps\to 0$ minimizers $\overline{\E_\eps}$ of $P_\eps$ converge up to subsequence and rigid motions to a cluster of disks that maximizes the number of contact points among the disks, each contact point counted with factor $\frac{r_ir_j}{r_i+r_j}$, where $r_i,r_j$ are the radii of the touching disks.
\end{theorem}

\begin{remark}
Theorem \ref{thm:sticky} selects, among all possible clusters of disks with the right area constraint, those which maximize the number of (weighted) tangencies; equivalently, those which minimize the following \emph{tangency functional}
\begin{equation}\label{eq:tangency}
\T(\E)=\begin{cases}-\displaystyle\!\!\sum_{1\leq i<j\leq N} \sigma_{ij}\frac{2r_ir_j}{r_i+r_j} & \text{ if $\E=(B_1,\ldots, B_N)$ is a cluster of disks}\\
+\infty & \text{ otherwise }
\end{cases}
\end{equation}
where $r_i$ is the radius of the disk $B_i$ and
\[\sigma_{ij}=\begin{cases}1&\text{ if $B_i$ and $B_j$ touch }\\
0& \text{ otherwise}
\end{cases}.
\]
In the case of equal radii, the tangency functional $\T$ coincides, up to a suitable rescaling, with the energy of $N$ particles associated to the centers of $B_i$ and interacting by means of the \emph{sticky disk} (or \emph{Heitmann-Radin}) potential
\[
V(r)=\begin{cases}
+\infty & \text{if } r<1\\
-1 & \text{if } r=1\\
0 & \text{if } r>1
\end{cases},
\]
hence the name of the Theorem above. Heitmann and Radin proved in \cite{HR80} that minimizers for the sticky disk with a fixed number of particles $N$ are crystallized, that is they form a subset of the triangular lattice. Moreover as $N\to\infty$ the global shape of minimizers converges to a hexagon \cite{AFS12}, \cite{Sch13}, \cite{DPS16}. In view of Theorem \ref{thm:sticky} this translates in the context of $N$-clusters minimizing $P_\eps$ in the following information: if we first send $\eps\to 0$ and then $N\to\infty$ we obtain as an asymtptotic global shape a hexagon. If it were possible to exchange the order of the limits we would obtain that, for sufficiently small $\eps$, the global shape of $N$-clusters minimizing $P_\eps$ is almost hexagonal in the limit $N\to\infty$. This would give a partial answer in the case of weighted clusters to a question considered by Cox, Morgan and Graner \cite{CMG13} about the global shape of minimal $N$-clusters for large $N$, and it was actually the initial motivation for this work.
\end{remark}

The main ingredient in the proof of Theorem \ref{thm:sticky} is the lower-bound inequality given by Theorem \ref{thm:isop}, which can be seen as an asymptotic quantitative isoperimetric inequality involving the “curvature deficit” of the boundary.

Finally, as a byproduct of the proof of Theorem \ref{thm:sticky}, we also obtain information on the structure of minimizers $\overline{\E_\eps}$ for small $\eps$:

\begin{theorem}[Structure of minimizers]\label{thm:structure}
Minimizing clusters $\overline{\E_\eps}$ have the following properties: let $\B=(B_1,\ldots,B_N)$ be a cluster of disks with radii $r_1,\ldots ,r_N$ to which $\overline{\E_\eps}$ converge; then for small $\eps>0$, in addition to the standard regularity given by Theorem \ref{thm:regularity}, the following hold: 

\begin{itemize}[label=\raisebox{0.25ex}{\tiny$\bullet$}]
\item each chamber is connected;
\item different arcs can meet only in a finite number of triple points, and when this happens exactly one of the chambers meeting there is the exterior one. In particular, the angles formed at a triple point are $2\theta_\eps,\pi-\theta_\eps,\pi-\theta_\eps$, where $\theta_\eps=\arccos\left(1-\frac{\eps}{2}\right)$.
\item between each pair of chambers $\overline{\E_\eps}(i)$ and $\overline{\E_\eps}(j)$ such that $B_i$ and $B_j$ are tangent, there is a single arc of constant curvature $\kappa_{ij}^\eps$ and length of respective chord $\ell_{ij}^\eps$ where
\[
\kappa_{ij}^\eps=\frac12\left(\frac{1}{r_j}-\frac{1}{r_i}\right)+o(1)\qquad\text{and}\qquad
\ell_{ij}^\eps=\frac{4r_i r_j}{r_i+r_j}\eps^{1/2}+o(\eps^{1/2})
\]
while in the remaining portion of the boundaries, that is between any chamber $\overline{\E_\eps}(i)$, $i\geq 1$, and the exterior  $\overline{\E_\eps}(0)$, there is an arc of curvature $\kappa_i^\eps=\frac{1}{r_i}(1+o(1))$.
\end{itemize} 
\end{theorem}

\begin{remark} \emph{$\Gamma$-convergence.} We decided to state Theorem \ref{thm:sticky} talking about minimizers, but actually a stronger result holds: the rescaled functionals $P^{(1)}_\eps$ given by \eqref{eq:P1eps} $\Gamma$-converge to the tangency functional $\T$ given by \eqref{eq:tangency}, with respect to the $L^1$-convergence of clusters (we refer to \cite{Bra02} for the definition and the properties of $\Gamma$-convergence). The hard part is the $\liminf$ inequality: to prove it, given any family $\E_\eps$ converging to a cluster of disks $\B$, we can build an improved family with a higher regularity using for instance the density of polygonal clusters among all clusters \cite{BCG17}, and then apply Theorem \ref{thm:isop}. The method of looking at the second order behaviour of $P_\eps$ is close in spirit to \cite{AnzBal93}.
\end{remark}

To conclude, we briefly outline the structure of this article. In Section \ref{sec:prelim} we introduce the notation, recall basic facts about minimal clusters and prove preliminary results. In Section \ref{sec:firstorder} we show the first-order result of Proposition \ref{prop:round}. In Section \ref{sec:secondorder} we prove Theorem \ref{thm:sticky} and Theorem \ref{thm:structure}. In particular in Subsections \ref{subsec:localization} and \ref{subsec:contacts} we prove that for $\eps$ small enough each chamber of a minimizer is connected (Lemma \ref{lemma:localizationsimple}) and that there is at most one boundary arc between two different chambers (Lemma \ref{lemma:onearc}). In Subsection \ref{sec:isop} we prove an asymptotic version of quantitative isoperimetric inequality, where the isoperimetric deficit controls the “curvature deficit” of the boundary. From this result we deduce the key lower bound for the perimeter of a given cluster converging to a cluster of disks (Proposition \ref{prop:keyineq}). Finally, in Subsection \ref{subsec:recovery} we build a \emph{recovery sequence} for Theorem \ref{thm:sticky}, that is we prove that the previous lower bound is sharp, and then prove the theorems (Subsection \ref{subsec:proofs}). We conclude with some remarks (Section \ref{sec:remarks}).

\section{Notation and preliminary results}\label{sec:prelim}

\paragraph{\thesection.1 Definitions}
We use the notation $f(\eps)=O(g(\eps))$ and $f(\eps)=o(g(\eps))$ to mean respectively 
\[
\limsup_{\eps\to 0^+}\tfrac{|f(\eps)|}{g(\eps)}<\infty\quad \text{  and }\quad \lim_{\eps\to 0^+}\tfrac{f(\eps)}{g(\eps)}=0.
\]
We denote the area (Lebesgue measure) of a set $E\subset \R^2$ by $|E|$. A planar \emph{$N$-cluster}, or simply \emph{cluster} if the dependence on $N$ is clear, is a family $\E=\{\E(1),\ldots,\E(N)\}$ of disjoint nonempty open sets with finite area and piecewise smooth boundary. The sets $\E(i)$ are called  \emph{chambers} of the cluster (or also \emph{bubbles}, whence the name $N$-bubble), and are not required to be connected. It is useful to define also the \emph{exterior chamber} $\E(0)=\R^2\setminus \bigcup_{i=1}^N \E(i)$. The \emph{interface} between the chambers $\E(i)$ and $\E(j)$ is
\begin{equation}\label{eq:boundary}
\E(i,j):=\partial \E(i)\cap \partial \E(j) .
\end{equation}
The \emph{weighted perimeter} of a cluster is given by the weighted sum of the length of its interfaces as in \eqref{eq:P}.

It is useful to introduce a notion of convergence for $N$-clusters, namely $\E_k\to\E$ iff $|\E_k(i)\Delta \E(i)|\to 0$ for every $i=1,\ldots,N$, where $\Delta$ is the symmetric difference of sets (equivalently, the characteristic functions of each chamber converge in $L^1$). With respect to this convergence, the perimeter given by \eqref{eq:P} is lower semicontinuous if and only if the following triangle inequalities are satisfied: 
\begin{equation}\label{eq:wetting}
c_{ij}\leq c_{ik}+c_{kj} \quad\text{ for every choice of distinct $i,j,k$ }.
\end{equation}
For a reference see \cite{AB90II}, in particular Example 2.8 with $\psi\equiv 1$.

\paragraph{\thesection.2 Existence and regularity of minimal clusters}
We here recall the basic existence and regularity results for minimizing clusters in the plane, which can be found for instance in \cite{Mor98}. The existence of minimal $N$-clusters for a given area constraint follows by the direct method, and requires first to enlarge the class of competitors to include clusters made of finite perimeter sets and prove existence inside this class, and then to recover regularity of minimizers (we refer to \cite{Mag12} for an introduction on finite perimeter sets and clusters). We briefly recall here the basic definitions in this more general setting (for simplicity in the planar case), although we will be dealing with minimizers and thus only with sets of piecewise smooth boundary (in fact, piecewise of constant curvature).

A measurable set $E$ in $\R^2$ is said to be of finite perimeter if
\[P(E):=\sup \left\{\int_{E}\div \phi: \phi\in C^\infty_c(\R^2,\R^2), |\phi|\leq 1 \right\}<\infty.
\]
When the set $E$ is sufficiently regular, $P(E)=\mathrm{length}(\partial E)$. For a finite perimeter set $E$ it is useful to introduce the notion of \emph{essential boundary} $\partial^*E$, which is the set of points in the plane with Lebesgue density neither $0$ nor $1$. The essential boundary coincides with the topological boundary for regular sets. By the structure theorem of finite perimeter sets (\cite[Theorem 15.9]{Mag12}), $P(E)=\H^1(\partial^* E)$, where $\H^1$ is the $1$-dimensional Hausdorff measure. The notion of cluster in this setting can be given almost exactly as in the regular case: a planar $N$-cluster is a family $\E=\{\E(1),\ldots,\E(N)\}$ of finite perimeter sets such that 
\[0<|\E(i)|<\infty \text{ for $1\leq i\leq N$ };\qquad |\E(i)\cap\E(j)|=0 \text{ for $0\leq i< j\leq N$}.
\]
The perimeter of a cluster is obtained by replacing in \eqref{eq:boundary} and \eqref{eq:P} the topological boundary $\partial \E(i)$ with the essential boundary $\partial^* \E(i)$ and the length with the Hausdorff measure $\H^1$.

We note here for future reference that the functional $P_\eps$ can be rewritten in the following equivalent way:
\begin{equation}\label{eq:Pepsrewriting}
P_\eps(\E)= \left(1-\frac{\eps}{2}\right)\sum_{i=1}^N P(\E(i))+\frac{\eps}{2}P(\E(0)).
\end{equation}

By a standard compactness theorem for finite perimeter sets and the lower semicontinuity of the functional $P_\eps$ (which can be proved either checking that the triangle inequalities \eqref{eq:wetting} hold, or using \eqref{eq:Pepsrewriting} and the lower semicontinuity of the perimeter on each chamber) we can prove existence of a minimizer for $P_\eps$ (see also \cite[Section~3.3]{Mor98}):

\begin{theorem}[Existence]
For every $\eps\in [0,2]$ there is a minimizer for $P_\eps$ with any given volume constraint.
\end{theorem}

Regarding regularity of minimizers, we have the following theorem:
\begin{theorem}[\cite{Mor98}, Proposition 4.3]\label{thm:regularity}
Any minimizer $\overline{\E_\eps}$ of $P_\eps$ has the following properties:
\begin{itemize}[label=\raisebox{0.25ex}{\tiny$\bullet$}]
\item each chamber has a piecewise $C^1$ boundary made of a finite number of arcs with constant curvature;
\item these arcs meet in a finite number of vertices, where they satisfy the condition
\begin{equation}\label{eq:tripointcondition}
\sum_{} c_{ij}\tau_{ij}=0 
\end{equation}
where $\tau_{ij}$ is the unit vector starting from the vertex and tangent to $\partial \E(i)\cap \partial \E(j)$, and the sum is extended over all interfaces meeting at the vertex; 
\item  around any vertex the weighted curvatures sum to zero.
\end{itemize}
\end{theorem}

In the case where all weights are equal, something more can be said: namely that at each vertex exactly three arcs meet forming 120-degree angles. In the general case of minimal weighted clusters there could be also quadruple points (for instance consider four equal squares with a vertex in common, with weights $>\sqrt 2$ between diagonally-opposite squares and $1$ otherwise; this cluster is minimizing among clusters with the same boundary condition. Compare also with the example at the end of \cite[Section 2.3]{AB90II}). However, for our specific choice of weights given by \eqref{eq:Peps}, we are able to recover the triple-point property: exactly three arcs meet at each vertex, as the next lemma shows. This property should in principle be inferable from the algebraic conditions that weights have to satisfy at each vertex given in \cite[Remark 4.4]{Mor98}, however we prefer the following more direct and geometric argument.

\begin{lemma}[Triple-point property]\label{lemma:triplepoints}
For $\eps$ small enough, at every vertex of a minimizer of $P_\eps$ exactly three arcs meet. Moreover at every such vertex exactly one of the chambers is the exterior one $\E(0)$ and the angles are given by $\pi-\theta_\eps,\pi-\theta_\eps,2\theta_\eps$, where
\[\theta_\eps=\arccos\left(1-\frac{\eps}{2}\right).\]
\end{lemma}

\begin{proof}
We suppose that there is a vertex at which at least four arcs meet, and prove that the cluster is not minimal since we can modify it to lower the energy. We give the proof under the simplifying assumption  that the arcs meeting at the vertex are straight edges; the proof in the general case is almost identical, it suffices to zoom at a sufficiently small scale and apply the same argument. 

First we show that there can not be any component of the exterior chamber around such a point: 
\begin{itemize}[label=\raisebox{0.25ex}{\tiny$\bullet$}]
\item if there is only one component of the exterior chamber then, since at least one of the remaining angles is less than $120$ degrees, we could put a Steiner configuration inside a small triangle of small lengthscale $\delta$, fixing the area somewhere else (see Figure \ref{fig:vertex1});
\item if instead there are at least two components of the exterior chamber, then one of the remaining portions is contained in a half-plane. We can modify all the chambers in this half-plane removing completely a small triangle of small lengthscale $\delta$, and fix the area somewhere else (see Figure \ref{fig:vertex2}). 
\end{itemize}
In both cases, when $\delta$ is small enough, we reduce the perimeter since the reduction in perimeter due to the first modification is of order $\approx \delta$, while the change in perimeter due to the area-fixing variations is of order $\approx \delta^2$.

\begin{figure}
    \centering
    \begin{minipage}{0.45\textwidth}
        \centering
        \includegraphics[width=0.9\textwidth]{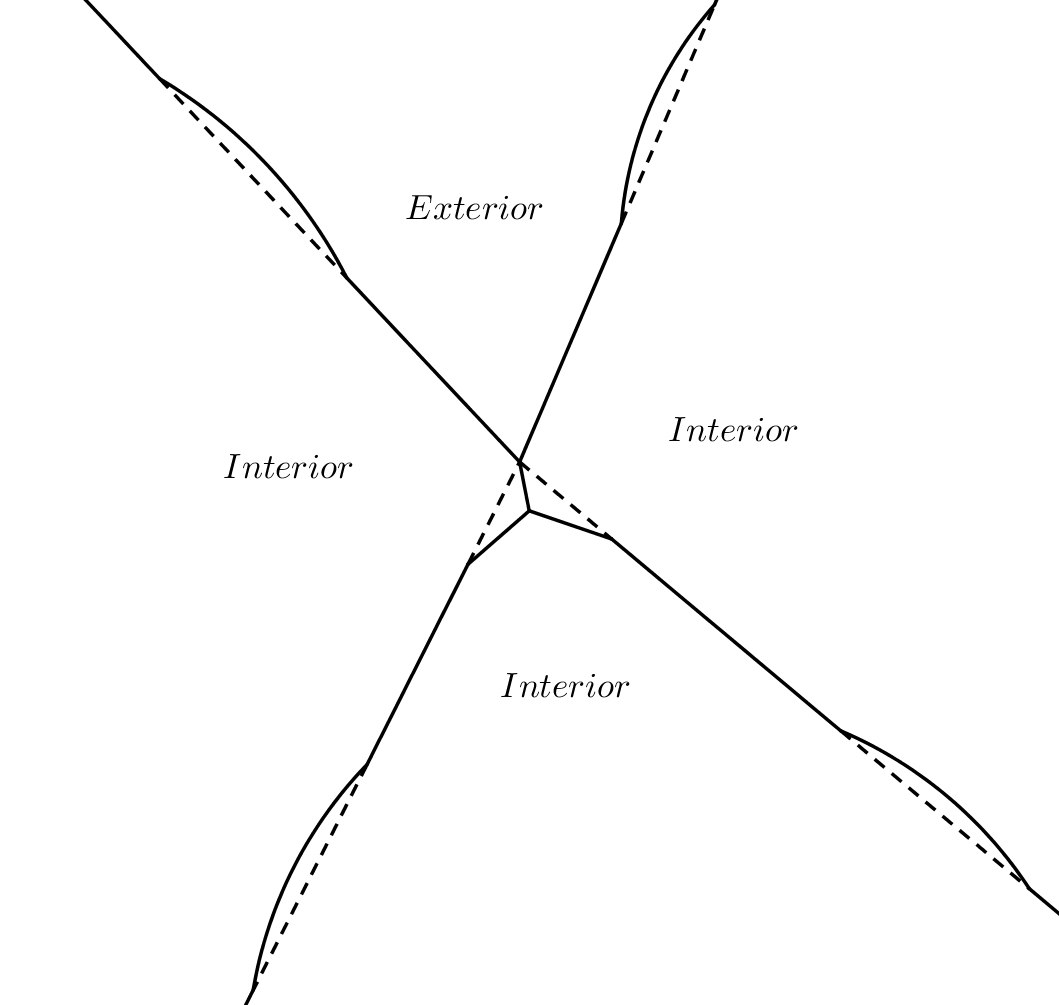}
        \caption{\small{If there is only one exterior component, we can put a small Steiner configuration in one of the remaining angles which is less than $120$ degrees. We have to fix the area with a slight inflation or deflation.}}\label{fig:vertex1}
    \end{minipage}\hfill
    \begin{minipage}{0.45\textwidth}
        \centering
        \includegraphics[width=0.96\textwidth]{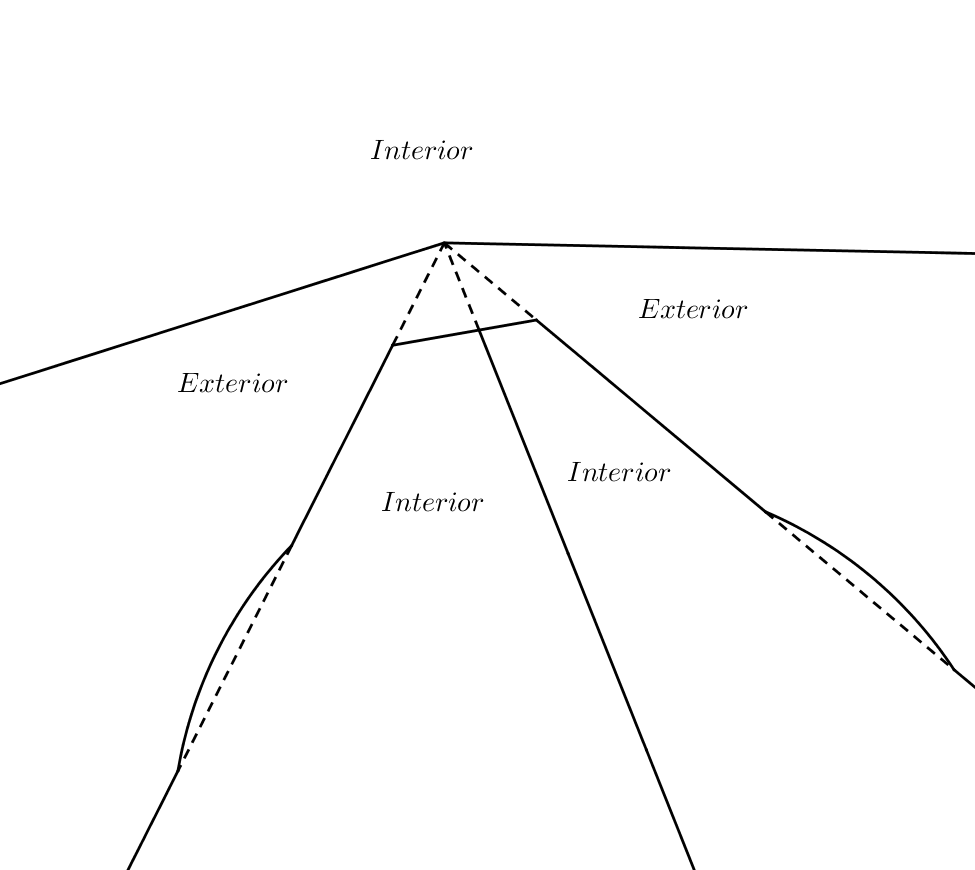}
        \caption{\small{If there are at least two exterior components, one of the remaining portions is contained in a half-plane. We can cut and remove a whole triangle, again fixing the area with a slight inflation or deflation.}}\label{fig:vertex2}
    \end{minipage}
\end{figure}

We are therefore left with a configuration in which there is no exterior chamber. But then, since we are supposing to have at least four components, at least one of the angles is less than $120$ degrees, and we can lower the energy again by putting a small Steiner configuration inside a small triangle. This proves that there must be exactly three arcs meeting at each vertex. The same proof, as already said, holds even if the arcs are curved, looking at a sufficiently small scale around the vertex and applying similar variations.

Now let us prove that around any vertex exactly two interior and one exterior components meet. If the three chambers meeting at a vertex were all interior chambers, the standard variational argument would imply that the angles are $120$ degrees; but then we could insert a small triangular hole (a component of the exterior chamber) around the vertex, again adjusting the area somewhere else. The reduction of perimeter is again of order $\delta$, plus corrections of order $\delta^2$ for the area adjustments. The key point is that the perimeter of an equilateral triangle is smaller than the length of its Steiner configuration multiplied by $2-\eps$, for $\eps$ sufficiently small. Therefore we conclude that the only components we can have are two interior chambers and one exterior chamber.

Finally, the computation of the angle $\theta_\eps$ comes directly from condition \eqref{eq:tripointcondition}.
\end{proof}

\paragraph{\thesection.3 Isoperimetric inequality}
We end this section by stating the isoperimetric inequality in the following form:

\begin{proposition}[Isoperimetric inequality with signed areas]\label{prop:isop}
A circumference enclosing area $m>0$ minimizes $\mathrm{length}(\gamma)$ among all oriented planar rectifiable curves $\gamma$ enclosing a signed area $m$.
\end{proposition}

\section{First order analysis: convergence to disks}\label{sec:firstorder}
In this section we prove the first-order result that the limit clusters are made of disks. We begin with a simple compactness result:
\begin{lemma}[Compactness]\label{lemma:compactness}
Any sequence of minimizers $\overline{\E_\eps}$ has uniformly bounded diameter, that is
\[\diam(\overline{\E_{\eps}})\leq C<+\infty.
\]
\end{lemma}

\begin{proof}
The result follows essentially from the fact that for connected sets in the plane the perimeter controls the diameter, namely $\diam E\leq \frac12 P(E)$. Supposing that  $E_\eps:=\bigcup_{i=1}^N \overline{\E_\eps}(i)$ is connected we indeed obtain
\[2\diam E_\eps\leq P(E_\eps)\leq P_\eps(\overline{\E_\eps})\leq \sum_{i=1}^N P(B_i)
\]
which gives the desired conclusion.

Let us now prove that $E_\eps$ is connected. By the regularity result of Theorem \ref{thm:regularity} we know that for every $\eps$ each chamber of $\overline{\E_\eps}$ is equivalent to an open set which has piecewise $C^1$ boundary. If $E_\eps$ were disconnected, we could take two connected components and move them until they touch without changing the value of $P_\eps$. The cluster thus created would still be minimal but would have at least a quadruple point, contradicting Lemma \ref{lemma:triplepoints}. This concludes the proof.
\end{proof}

We can now prove the first-order result of Proposition \ref{prop:round}.

\begin{proof}[Proof of Proposition \ref{prop:round}]
Using $N$ disjoint disks as competitors we obtain that 
\[P(\overline{\E_{\eps}}) \leq P_{\eps}(\overline{\E_{\eps}})\leq \sum_{i=1}^N \sqrt{4\pi m_i}<+\infty.
\]
Moreover by Lemma \ref{lemma:compactness} the sequence has uniformly bounded diameter, and thus the following uniform bound holds for minimizers $\overline{\E_\eps}$:
\[\sup_\eps P(\overline{\E_\eps})+\diam(\overline{\E_\eps})<\infty.
\] 
By a standard compactness result about finite perimeter sets (see \cite[Theorem 12.26]{Mag12}), this implies that minimizers $\overline{\E_\eps}$ converge, up to subsequence and rigid motions, to a limit cluster $\E_0$ with the same area constraint. 
By \eqref{eq:Pepsrewriting}, we also obtain that
\begin{equation}\label{eq:trivial}
P(\overline{\E_\eps}(i))\leq \sqrt{4\pi m_i}(1+O(\eps)),
\end{equation}
and by lower semicontinuity of perimeter we obtain
\[P(\E_0(i))\leq \sqrt{4\pi m_i}.
\]
By the isoperimetric inequality, the unique minimizer of perimeter for a given area constraint is the disk, and therefore $\E_0(i)$ is a disk of area $m_i$.
\end{proof}

\section{Second order analysis: sticky-disk limit}\label{sec:secondorder}

We now want to obtain some more information about minimizers $\overline{\E_\eps}$ as $\eps\to 0$.
In the last section we saw that, up to translation, minimal clusters converge to a cluster of disks; this was a simple consequence of the isoperimetric inequality together with a compactness result. However, as already pointed out, we don't expect to see in the limit every cluster of disks: for instance Lemma \ref{lemma:compactness} suggests that at least the limit clusters must be connected. To understand what kind of clusters can arise as limits, we will perform a higher order expansion of the perimeter.

\subsection{Localization of contacts between different chambers}\label{subsec:localization}
In this subsection we prove a localization result that basically says that each chamber of a minimizer is sandwiched between two concentric disks $(1-o(1))B$ and $(1+o(1))B$, $B$ being a disk with the same area as the chamber. This can be seen as an improvement from the $L^1$ convergence of Proposition \ref{prop:round} to “uniform” convergence, or Hausdorff convergence of the boundaries. A consequence of this is that any pair of chambers whose limit disks are not touching will eventually share no boundary. Moreover we prove that for $\eps$ small enough each chamber of a minimizer is connected.

\begin{lemma}[Localization Lemma]\label{lemma:localizationsimple}
Suppose that a minimizer $\overline{\E_\eps}$ converges to the cluster of disks $\B=(B_1,\ldots, B_N)$. Then, for $\eps$ small enough, each chamber $\overline{\E_\eps}(i)$ is connected, and moreover
\[(1-o(1)))B_i\subset \overline{\E_\eps}(i)\subset (1+o(1))B_i.
\]
\end{lemma}

\begin{proof} We fix a chamber and denote it for simplicity just by $E$, and the disk by $B$. We will prove the lemma in four steps:
\begin{enumerate}[label=(\roman*)]
\item For $\eps$ small enough, $E$ has only one biggest (in terms of area) connected component $C_0$, which carries almost all the mass, i.e.
\[|C_0|\geq |E|(1-o(1)).
\]
In particular, if $\overline{\E_\eps}(i)$ converges to a disk $B$ as $\eps\to 0$, then $|C_0\Delta B|=o(1)$.
\item The convex hull $co(C_0)$ is sandwiched between two disks both converging to $B$ as $\eps\to 0$:
\[(1-o(1))B\subset co(C_0)\subset (1+o(1))B.
\]
\item For $\eps$ small enough the biggest connected component is in fact the only one, i.e. each chamber is connected.
\item The same conclusion as in $(ii)$ holds also for $C_0=E$ itself, namely
\[(1-o(1))B\subset E\subset (1+o(1))B.
\]
\end{enumerate}

$\mathrm{(i)}$ If there is just one connected component then we are done. Otherwise, denote by $C_0,C_1,\ldots$ the connected components of $E$ (indexed by at most countably many indices $i$, and ordered decreasingly in the area), and set $V_i:=|C_i|/|E|$ to be the normalized area of the connected component $C_i$ of $E$ . In this way $\sum_i V_i=1$. Set $M:=\max_i V_i$. If $M\leq \frac12$, by the isoperimetric inequality
\[P(E)=\sum_i P(C_i)\geq 2\sqrt{\pi}\sum_i\sqrt{|C_i|}= 2\sqrt{\pi}\sqrt{|E|}\sum_i\sqrt{V_i}\geq 4\sqrt{\pi} \sqrt{|E|}
\]
where we used that, since in this case $V_i\leq M\leq 1/2$, we have $\sqrt{V_i}\geq 2 V_i$. But by the trivial energy estimate \eqref{eq:trivial} we know that $P(E)\leq 2\sqrt{\pi}\sqrt{|E|}(1+o(1))$, so for $\eps$ small enough $M$ must be $>\frac12$, and in particular there is only one component with maximum area, $C_0$. In this case for every $i\geq 1$ we have $V_i\leq 1-M<\frac12$, and arguing as above we obtain that
\[P(E)= P(C_0)+\sum_{i\geq 1}P(C_i)\geq 2\pi\sqrt{M}+\sum_{i\geq 1} 2\pi \frac{V_i}{\sqrt{1-M}}= 2\pi\sqrt{M}+2\pi\sqrt{1-M}.
\]
Again from the energy estimate we know that each chamber has an isoperimetric deficit $o(1)$, therefore we obtain the condition
\[\sqrt{M}+\sqrt{1-M}\leq 1+o(1),
\] 
which together with the condition $\frac12 \leq M\leq 1$  easily implies that $M$ must be close to $1$, which translates to $|C_0|\geq |E|(1-o(1))$.

$\mathrm{(ii)}$ First we prove that $co(C_0)\supset (1-o(1))B$. Indeed, given a point $x\in B\backslash co(C_0)$ (if it exists, otherwise we are done), we can find a whole circular cap whose straigth segment passes through $x$ that is contained in $B\backslash co(C_0)$. The area of this circular cap is at least as big as the area of the circular cap whose straight segment is perpendicular to the radius through $x$. From point $(i)$ this area must be $o(1)$, and this easily imples the desired conclusion.

Next we prove that $co(C_0)\subset (1+o(1))B$. We use the following two standard facts for planar sets: 
\begin{enumerate}[label=(\roman*)]
\item the convex hull of an open connected set has smaller perimeter than the original set; 
\item among convex bodies in the plane, the perimeter is monotone increasing with respect to inclusion.
\end{enumerate}
From the first fact we obtain that $P(co(C_0))\leq P(B)(1+o(1))$. Now take any point $x\in co(C_0)\backslash B$. By convexity and since $co(C_0)\supset (1-o(1))B$, we obtain that $co(C_0)\supset co\big((1-o(1))B\cup \{x\}\big)$. From the second fact cited above the latter set must have smaller perimeter than $co(C_0)$, and this easily implies that $x\in (1+o(1))B$.

$\mathrm{(iii)}$ Suppose that $E$ has more than one component. From point $(ii)$ we know that all the components except for the biggest one $C_0$ have a total mass of at most $0<m=o(1)$. Then by the isoperimetric inequality and the subadditivity of the square root, their total perimeter is bigger than
\[\sum_{i\geq 1} 2\sqrt{\pi}\sqrt{|C_i|}\geq 2\sqrt{\pi}\sqrt{m}.
\]
We now remove all the smaller components, and inflate the biggest one, and prove that for $\eps$ small enough we find in this way a better competitor, which is incompatible with the supposed minimality of the original cluster. The increase in perimeter due to the inflation can be taken of order of the total removed mass $m$, see for instance \cite[Theorem~29.14]{Mag12}. The net change in perimeter is therefore $-2\sqrt\pi\sqrt m+b m$ for some constant $b$, which for $m>0$ small enough is negative; the same net change holds also for the functional $P_\eps$.  This proves that for $\eps$ small enough, and therefore $m$ small enough, there can be just one connected component for each chamber.

$\mathrm{(iv)}$ The rightmost inclusion follows immediately from $C_0\subset co(C_0)$ and point $(ii)$. We now prove the other one. From this last inclusion we know that the only obstacle would be the presence of the exterior chamber inside $B$. To exclude this we argue similarly to point $(iii)$: if there are connected components of the exterior chamber entirely surrounded by $E$, we can “fill” them with the set $E$, and then perform a deflation of the set, which for $m$ small enough results in a net decrease in the functional $P_\eps$. If instead there are “tentacles” of the exterior chamber which come from the outside, that is components not entirely surrounded by $E$, by similar considerations they must be contained in the complement of $(1-o(1))B$, and we are done.
\end{proof}

\subsection{There is eventually at most one contact between any pair of chambers}\label{subsec:contacts}

Next we shall prove that when $\eps$ is small enough, there is at most one arc in common between two different chambers.
\begin{definition}\label{def:equiv}
Given a set $E\subset \R^2$, we set $B_E$ as the disk of the same area (say, centered at the origin), $r_E:=\sqrt{|E|/\pi}$ as its radius and $\kappa_E=1/r_E$ as the curvature of $\partial B_E$.
\end{definition}
Recall that the interface between the chambers $\E_\eps(i)$ and $\E_\eps(j)$ is $\E_\eps(i,j)=\partial \E_\eps(i)\cap \partial \E_\eps(j)$, and that $\overline{\E_\eps}$ denotes a minimizer for $P_\eps$.

\begin{lemma}\label{lemma:kappabound} 
The curvature of the interface arcs $\overline{\E_\eps}(i,j)$ converges up to sign as $\eps\to 0$ to:
\begin{enumerate}[label=(\roman*)]
\item $\kappa_{\E(i)}$ if $j=0$;
\item $\frac12(\kappa_{\E(i)}-\kappa_{\E(j)})$ if $i,j\neq 0$.
\end{enumerate} 
\end{lemma}

\begin{proof}
By Theorem \ref{thm:regularity} (regularity) we know that the weighted curvatures sum to zero around any vertex:
\[
(2-\eps)\kappa_{ij}^\eps+\kappa_{j0}^\eps+\kappa_{0i}^\eps=0.
\]
It is therefore sufficient to prove $(i)$. This follows from the localization lemma \ref{lemma:localizationsimple}: since each chamber $\overline{\E_\eps}(i)$ is sandwiched between two concentric disks whose radii converge to the same value as $\eps \to 0$, contacts between different chambers can happen in a finite number of zones whose diameter converge to zero. In the complement of these zones there will be only arcs of constant curvature, without triple points. Since each one of these arcs is sandwiched between two concentric disks converging to the same disk, the curvature must converge to the limit curvature $\kappa_{\E_\eps(i)}=\kappa_{\B(i)}$.
\end{proof}

\begin{lemma}\label{lemma:onearc}
\leavevmode
\begin{enumerate}[label=(\roman*)]
\item The length of every interface between any pair of chambers goes to $0$ as $\eps\to 0$, that is
\[\lim_{\eps\to 0}\H^1(\overline{\E_\eps}(i,j))=0;
\]
\item for $\eps$ small enough, any pair of chambers of $\overline{\E_\eps}$ share at most one arc, that is $\overline{\E_\eps}(i,j)$ has at most one connected component. If the two chambers converge to non-tangent disks, then they eventually share no boundary.
\end{enumerate}
\end{lemma}

\begin{proof}
$\mathrm{(i)}$ 
This is a consequence of the localization lemma \ref{lemma:localizationsimple} and the lower semicontinuity of the perimeter. If two chambers converge to two non tangent disks, then the interface is eventually empty by the localization lemma and we are done. Otherwise, consider the case where the two limit disks have a tangency point $p$, and suppose by contradiction that for a sequence $\eps_h\to 0$ it holds $\H^1(\overline{\E_{\eps_h}}(i,j))\geq c>0$. Notice that again by the localization lemma, the interface is contained in a curved wedge that as $\eps\to 0$ converges to the point $p$. Since $\overline{\E_\eps}(i)\to B_i$, for every closed neighbourhood $K$ of $p$ we have by semicontinuity
\[P(B_i,K^c)\leq \liminf_{h\to\infty} P(\overline{\E_{\eps_h}}(i),K^c).
\]
Adding the inequality 
\[c\leq \H^1(\overline{\E_{\eps_h}}(i,j))\leq \liminf_{h\to\infty} P(\overline{\E_{\eps_h}}(i),K)
\]
we obtain
\begin{align*}
P(B_i,K^c)+c& \leq \liminf_{h\to\infty} P(\overline{\E_{\eps_h}}(i),K^c) +\liminf_{h\to\infty} P(\overline{\E_{\eps_h}}(i),K)\\
&\leq \liminf_{h\to\infty} P(\overline{\E_{\eps_h}}(i))\\
&=P(B_i)
\end{align*}
which yields a contradiction by choosing the neighbourhood $K$ small enough.

$\mathrm{(ii)}$ Suppose there is a component $C$ of the exterior chamber entirely surrounded by two other chambers $A$ and $B$. We prove that it is more convenient to add this component to one of the chambers and fix its total volume with a slight deflation. Call $\ell_A$ and $\ell_B$ the length of the respective interfaces with $C$, and suppose $\ell_A\leq \ell_B$. Then add $C$ to chamber $B$, and slightly deflate $B$ far from contact zones (which is always possible for small $\eps$). The contributions to $P_\eps$ coming from $C$ change from $\ell_A+\ell_B$ to $(2-\eps)\ell_A$, with a total change of $\ell_A-\ell_B-\eps\ell_A<0$, while the deflation to fix the total volume of $B$ can be chosen so that the energy decreases; this results in a global decrease in the energy $P_\eps$. 
\end{proof}

\subsection{An asymptotic quantitative isoperimetric inequality involving curvature}\label{sec:isop}
The aim of the following theorem is to prove a particular instance of quantitative isoperimetric inequality in the plane, involving how much the curvature of the boundary of a given set $E$ deviates on small scales from the “ideal” curvature $\kappa_E$.

\begin{theorem}\label{thm:isop}
Let $E\subset\R^2$ be open, of finite area and perimeter and let $\bar\kappa>0$ be a real number. Suppose the boundary of $E$ contains $m\in\N$ portions made of arcs with constant curvature $\kappa_1,\ldots,\kappa_m$, with $\kappa_i\leq \bar\kappa$, each arc having a corresponding chord of length $\ell_i$. The curvature is signed, meaning that it is positive if the arc is curved outwards, and negative if it is curved inwards. Then
\[P(E)\geq \sqrt{4\pi|E|}+\frac{1}{24}\sum_{i=1}^m \ell_i^3 (\kappa_i-\kappa_E)^2-O\left(\sum_{i=1}^m\ell_i^5\right).
\]
\end{theorem}


We begin with a simple lemma, of which we omit the proof.

\begin{figure}
\centering
\includegraphics[width=0.4\textwidth]{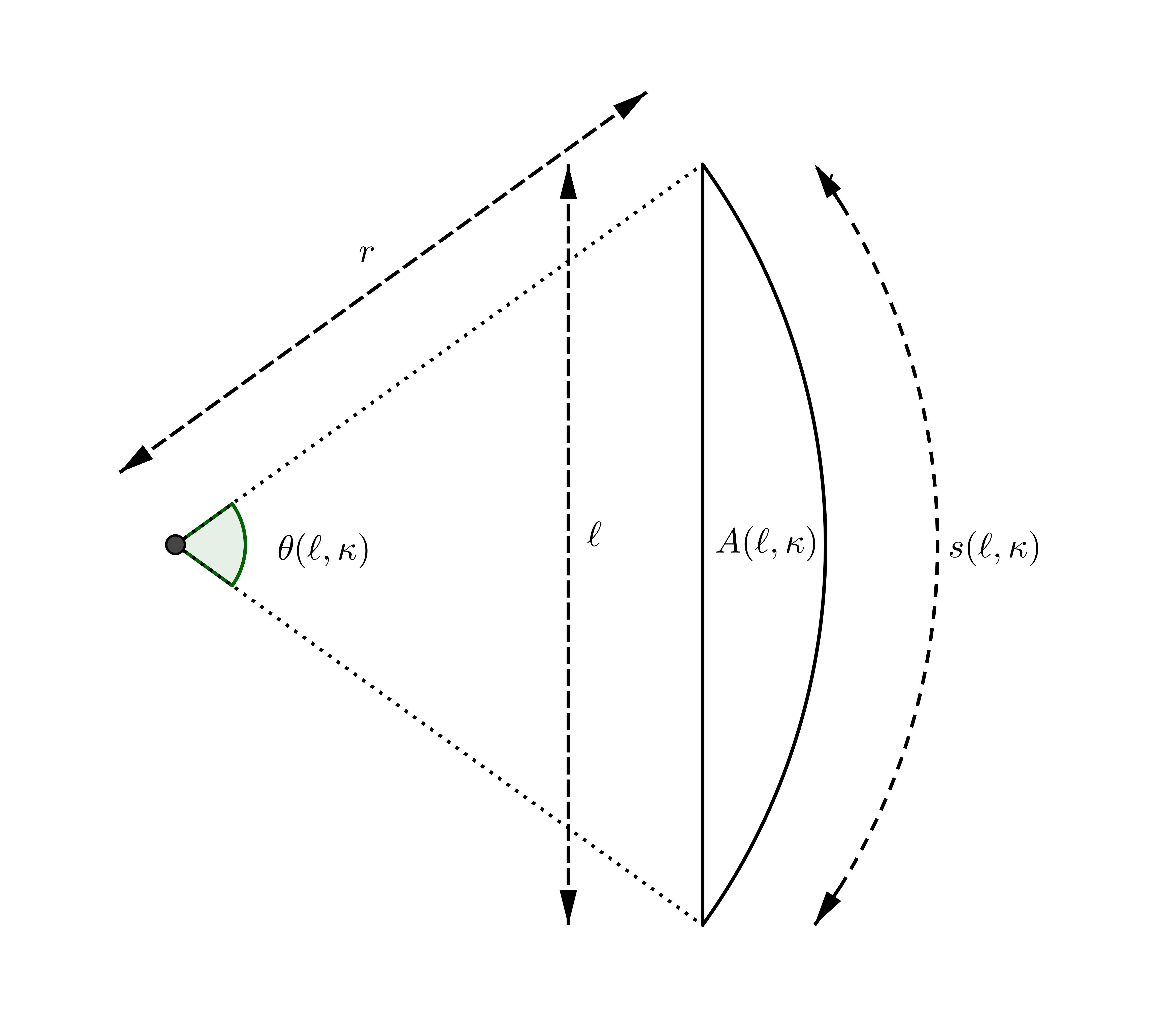}
\caption{\small{Reference figure for Lemma \ref{lemma:explicit}.}}\label{fig:explicit}
\end{figure}

\begin{lemma}\label{lemma:explicit}
Consider a segment in the plane of length $\ell$ and an arc of constant curvature $\kappa$ connecting its endpoints, and let $\theta$ and $r$ be the related angle and radius as in Figure \ref{fig:explicit}. Then the angle $\theta$, the length $s$ of the arc and the area $A$ of the circular section are given respectively by:
\begin{align*}
\theta(\ell,\kappa)& =2\arcsin{\frac{\ell \kappa}{2}} =\ell\kappa+\frac{1}{24}\ell^3\kappa^3 +O(\ell^5\kappa^5)\\
s(\ell,\kappa)&=\frac{2}{\kappa}\arcsin\left(\frac{\ell \kappa}{2}\right)=\ell+\frac{1}{24}\ell^3\kappa^2+O(\ell^5\kappa^3)\\
A(\ell,\kappa)&=\frac{\theta r^2}{2}-r^2\cos\frac{\theta}{2}\sin\frac{\theta}{2}=\frac{1}{12}\ell^3\kappa +O(\ell^5\kappa^3).
\end{align*}
\end{lemma}

We now pass to the proof of Theorem \ref{thm:isop}.

\begin{proof}[Proof of Theorem \ref{thm:isop}]
We inflate or deflate each arc until it reaches curvature $\kappa_E$, that is we replace the given arcs of curvature $\kappa_i$ with an arc of curvature $\kappa_E$ with the same endpoints, to obtain a new set $\tilde E$ with area $A(\tilde E)=A(E)+\Delta A$ and perimeter $P(\tilde E)=P(E)+\Delta P$; then we apply the isoperimetric inequality of Proposition \ref{prop:isop} to the new set $\tilde E$ and draw the consequences for the original set $E$, exploiting the explicit fomulas given by the previous lemma. We set for simplicity $\ell=\sum_{i=1}^m\ell_i$. By Lemma \ref{lemma:explicit} we can compute explicitly

\noindent\begin{minipage}{.5\linewidth}
\begin{align*}
\Delta P&=\sum_{i=1}^m \left(s(\ell_i,\kappa_E)-s(\ell_i ,\kappa_i)\right)\\
&=\sum_{i=1}^m\frac{1}{24} \ell_i^3(\kappa_i^2-\kappa_E^2)+O\left(\ell^5\right)
\end{align*}
\end{minipage}
\begin{minipage}{.5\linewidth}
\begin{align*}
\Delta A&=\sum_{i=1}^m\left( A(\ell_i,\kappa_E)-A(\ell_i, \kappa_i)\right)\\
&=\sum_{i=1}^m\frac{1}{12}\ell_i^3(\kappa_E-\kappa_i)+O\left(\ell^5\right).
\end{align*}
\end{minipage}
The isoperimetric inequality applied to $\tilde E$ gives $P(\tilde E)\geq \sqrt{4\pi|\tilde E|}$. Therefore
\begin{align*}
P(E) =P(\tilde E)-\Delta P &\geq \sqrt{4\pi} \sqrt{|E|+\Delta A} -\Delta P\\
	& = \sqrt{4\pi |E|}\sqrt{1+\frac{\Delta A}{|E|}}-\Delta P\\
	& =\sqrt{4\pi |E|}\left(1+\frac12\frac{\Delta A}{|E|}\right)+O\left(\frac{\Delta A}{|E|}\right)^2-\Delta P\\
	& = \sqrt{4\pi|E|} +\kappa_E \Delta A-\Delta P+O(\ell^6).
\end{align*}
Inserting now the asymptotic expansions for $\Delta A$ and $\Delta P$ we obtain
\begin{align*}
P(E)& \geq \sqrt{4\pi|E|}+\sum_{i=1}^m\left(\kappa_E\frac{1}{12} \ell_i^3 (\kappa_E-\kappa_i)-\frac{1}{24}\ell^3(\kappa_e-\kappa_i)^2\right)-O\left(\ell^5\right)\\
	& = \sqrt{4\pi |E|}+\sum_{i=1}^m\frac{1}{24}\ell_i^3(\kappa_i-\kappa_E)^2-O\left(\ell^5\right).
\end{align*}
\end{proof}

\subsection{Consequences for the \texorpdfstring{$N$}{N}-bubble: lower-bound inequality}\label{sec:liminf}

We will now draw the consequences of Theorem \ref{thm:isop} in the general case of weighted clusters with possibly different areas, obtaining the lower bound for the energy $P_\eps$ given by Proposition \ref{prop:keyineq}. We find it useful, however, to examine first the simpler case of a double bubble with equal areas, to explain the idea behind it. In particular, we will obtain the asymptotics given by \eqref{eq:explicit} as a lower bound, without using the explicit shape of minimizers.

\begin{proposition}
For every $2$-cluster $\E=(E_1,E_2)$ with both areas equal to $\pi$ we have
\[P_\eps(\E)\geq 4\pi-\frac43 \eps^{3/2}-O(\eps^{5/2}).
\]
\end{proposition}

\begin{proof} It is clearly sufficient to prove the statement when $\E$ is a minimizer of $P_\eps$ under the same volume constraint.
By Lemma \ref{lemma:onearc} we know that the chambers will have at most one single arc in common. Suppose this arc has length $s$ and curvature $\kappa$,  and that the chord of this arc has length $\ell$ . Then writing 
\[P_\eps(\E)=P(E_1)+P(E_2)-\eps s,
\]
recalling Lemma \ref{lemma:explicit} and applying Theorem \ref{thm:isop} to both chambers we obtain
\begin{align*}
P_\eps(\E)&\geq 4\pi+\frac{1}{24}\ell^3\left((1-\kappa)^2+(1+\kappa)^2\right)-\eps \left(\ell+\frac{1}{24}\ell^3\kappa^2 +O(\ell^5)\right) -O(\ell^5)\\
& = 4\pi+\frac{1}{12}\ell^3-\eps \ell+\frac{1}{12}\kappa^2 \ell^3\left(1-\frac{\eps}{2}\right)-O(\ell^5)\\
& \geq 4\pi+\frac{1}{12}\ell^3-\eps \ell-O(\ell^5)
\end{align*}
where the key fact is that the curvature $\kappa$ appears in the first line once with a positive sign and once with a negative sign, and where the last inequality follows from $\eps\leq 2$. We now optimize in $\ell\geq 0$ the expression $\tfrac{1}{12} \ell^3-\eps \ell$ to obtain the minimum point $\ell=2\eps^{1/2}$, and thus obtaining
\[P_\eps(\E)\geq 4\pi -\frac43 \eps^{3/2}-O(\eps^{5/2})
\]
as wanted.
\end{proof}

We will now perform a computation similar to the previous one, but this time for a general number $N$ of chambers and possibly different areas, to obtain a lower bound for the energy $P_\eps$.

\begin{proposition}\label{prop:keyineq}
Let $\E=\{E_1,\ldots,E_N\}$ be a planar cluster whose chambers have areas $|E_i|=m_i=\pi r_i^2$ and therefore ideal curvature $\kappa_{E_i}=1/r_i$ (see Definition \ref{def:equiv}), and whose boundaries have piecewise constant curvature. Suppose that every pair of chambers shares at most one arc. 
Then
\begin{equation}\label{eq:keyineq}
P_\eps(\E)\geq \sum_{i=1}^N P(B_{E_i})-\frac43 \eps^{3/2} \sum_{1\leq i<j\leq N} \sigma_{ij}\frac{2r_ir_j}{r_i+r_j}+O(\eps^{5/2})
\end{equation}
where $\sigma_{ij}$ is one if the chambers $E_i$ and $E_j$ share some boundary, and zero otherwise.
\end{proposition}

\begin{proof}
Call $\kappa_{ij}$ the curvature of the arc between chambers $i$ and $j$, $s_{ij}$ its length and $\ell_{ij}$ the length of the relative chord (we omit for simplicity the dependence on $\eps$), and set  $\ell=\sum_{i,j} \ell_{ij}$. We apply Theorem \ref{thm:isop} to each chamber to obtain
\begin{align*}
P_\eps(\E)&=\sum_{i=1}^N P(E_i)-\eps\sum_{1\leq i<j\leq N} s_{ij}\\
&\geq \sum_{i=1}^N \left(P(B_{E_i})+\frac{1}{24}\sum_{j\neq i} \ell_{ij}^3(\kappa_{ij}-\kappa_{E_i})^2-O(\ell^5)\right)-\eps\sum_{1\leq i<j\leq N} s_{ij}\\
& \begin{aligned}
=\sum_{i=1}^N P(B_{E_i})+\sum_{1\leq i<j\leq N} &\left(\frac{1}{24}\ell_{ij}^3 \left( (\kappa_{ij}-\kappa_{E_i})^2+(\kappa_{ij}+\kappa_{E_j})^2\right)\right.\\
& \left. \qquad-\eps\left(\ell_{ij}+\frac{1}{24}\ell_{ij}^3\kappa_{ij}^2\right)\right) -O(\ell^5).\end{aligned}
\end{align*}
Now we first optimize in $\kappa_{ij}$ each term in the sum, i.e. the quadratic polynomial in $\kappa_{ij}$ given by
\begin{align*}
& \frac{1}{24}\ell_{ij}^3 \left((\kappa_{ij}-\kappa_{E_i})^2+(\kappa_{ij}+\kappa_{E_j})^2\right)-\eps\left(\ell_{ij}+\frac{1}{24}\ell_{ij}^3\kappa_{ij}^2\right)\\
=& \frac{1}{24}\ell_{ij}^3\left((2-\eps)\kappa_{ij}^2+2(\kappa_{E_j}-\kappa_{E_i})\kappa_{ij}+\kappa_{E_i}^2+\kappa_{E_j}^2\right)-\eps \ell_{ij}.
\end{align*}
The minimum point is easily seen to be
\begin{equation}\label{eq:optimalkappa}
\kappa_{ij}=\frac{\kappa_{E_i}-\kappa_{E_j}}{2-\eps}
\end{equation}
giving the expression a minimum value of
\begin{align}
&\frac{1}{24}\ell_{ij}^3\left(\kappa_{E_i}^2+\kappa_{E_j}^2+\frac{(\kappa_{E_i}-\kappa_{E_j})^2}{2-\eps}\right)-\eps\ell_{ij}\nonumber\\
 =&\frac{1}{24}\ell_{ij}^3\left(\frac12(\kappa_{E_i}+\kappa_{E_j})^2-\frac{\eps}{4-2\eps}(\kappa_{E_i}-\kappa_{E_j})^2\right)-\eps\ell_{ij}.\label{eq:optimize}
\end{align}
We now optimize in $\ell_{ij}$: setting the derivative in $\ell_{ij}$ equal to zero we find
\begin{align*}
\ell_{ij}^2&=\frac{8\eps}{\frac12(\kappa_{E_i}+\kappa_{E_j})^2-\frac{\eps}{4-2\eps}(\kappa_{E_i}-\kappa_{E_j})^2}\\
& =\frac{16\eps }{(\kappa_{E_i}+\kappa_{E_j})^2}+O(\eps^{2})
\end{align*}
which implies
\begin{equation}\label{eq:optimalell}
\ell_{ij}=\frac{4}{\kappa_{E_i}+\kappa_{E_j}}\eps^{1/2}+O(\eps^{3/2}).
\end{equation}
Substituting this back into \eqref{eq:optimize} and observing that by the previous computation $O(\ell^5)=O(\eps^{5/2})$, we obtain that the expression is greater than
\[-\frac{8}{3}\eps^{3/2}\frac{1}{\kappa_{E_i}+\kappa_{E_j}}+O(\eps^{5/2})
\]
and now summing among all pairs $(i,j)$ we obtain
\[P_\eps(\overline{\E_\eps})\geq \sum_{i=1}^NP(B_{E_i})-\frac{4}{3}\eps^{3/2}\!\!\sum_{1\leq i<j\leq N}\sigma_{ij}\frac{2}{\kappa_{E_i}+\kappa_{E_j}}+O(\eps^{5/2})
\]
which is the desired result.

\end{proof}

As a consequence of the previous inequality and Lemmas \ref{lemma:localizationsimple} and \ref{lemma:onearc} we obtain the following:

\begin{corollary}\label{cor:keyineq}
Suppose minimizers $\overline{\E_\eps}$ converge as $\eps\to 0$ to the cluster of disks $\B$. Then
\[P(\overline{\E_\eps})\geq P(\B)-\frac43\eps^{3/2}\T(\B)+O(\eps^{5/2})
\]
where $\T$ is the tangency functional \eqref{eq:tangency}.
\end{corollary}

\begin{remark}[Non-optimal lower bound for \texorpdfstring{$P_\eps$}{Peps}]
Viewing an $N$-cluster as a “superposition” of $2$-clusters we can obtain a worse lower bound than equation \eqref{eq:keyineq}, but with the same order of $\eps^{3/2}$ for the second term. We notice that for $N\geq 2$ we can rewrite
\[P_\eps(\E)=\frac{1}{N-1}\sum_{1\leq i<j\leq N} P_{\delta(\eps)}\big((\E(i),\E(j))\big)
\]
where $\delta(\eps)=(N-1)\eps$ and $P_\delta((\E(i),\E(j)))=P(\E(i))+P(\E(j))-\delta \H^1(\E(i,j))$ is the weighted perimeter of the $2$-cluster $(\E(i),\E(j))$. From the solution of the double bubble (for simplicity in the case of equal volumes $|\E(i)|=\pi$) we know that $P_\delta((E_i,E_j))\geq 4\pi-\frac43\delta^{3/2}+O(\delta^{5/2})$ from which
\begin{align*}
P_\eps(\E)&=\frac{1}{N-1}\sum_{1\leq i<j\leq N} P_{\delta(\eps)}((E_i,E_j))\\
& \geq \frac{1}{N-1}\sum_{1\leq i<j\leq N}\left( 4\pi-\sigma_{ij}(\E)\frac43\delta^{3/2}+O(\delta^{5/2})\right)\\
& = 2N\pi -\frac43\sqrt{N-1}\C(\E)\eps^{3/2}+O(\eps^{5/2})
\end{align*}
where $\C(\E)=\sum_{i<j} \sigma_{ij}(\E)$ is the number of pairs $(i,j)$ such that $E_i$ and $E_j$ share some boundary. This is the estimate we are looking for, except for the factor $\sqrt{N-1}$ which makes the inequality worse. Observe that we can not obtain in this way the optimal inequality we are aiming to: indeed each double-bubble inequality is optimal when there is just one contact between two disks and the remaining portion of boundary is circular, which can not be  simultaneously true for all pairs of bubbles.
\end{remark}

\subsection{Sharpness of lower bound (recovery sequence)}\label{subsec:recovery}
We now want to show that the inequality proved in Corollary \ref{cor:keyineq} is essentially sharp, which means that, given a cluster of disks $\B=(B_1\ldots, B_N)$, we can actually find a sequence of clusters $\E_\eps$ converging to $\B$ for which the reverse inequality holds. We think that there should be a simpler way to do this other than the way proposed in the following, analyzing the sharpness of the inequality of Theorem \ref{thm:isop}, which is used to prove Proposition \ref{prop:keyineq}. However we were not able to follow this route and instead propose in the following a quite explicit and tedious computation for the polar equation of each chamber of an approximating sequence. 

The idea is to construct between any pair of tangent disks $B_i$, $B_j$ in the limit cluster $\B$ an arc whose constant curvature is $\frac12(\kappa_{E_i}-\kappa_{E_j})$ (which is the right asymptotic value given by condition $(ii)$ in Lemma \ref{lemma:kappabound}), of length $\ell_{ij}^\eps=4\eps^{1/2}/(\kappa_{E_i}+\kappa_{E_j})$ (which is up to $O(\eps^{3/2})$ the optimal value found in \eqref{eq:optimalell}). In the remaining portion of the boundaries of the chambers $\E_\eps(i)$ we can pretty much put any interface which, in polar coordinates w.r.t. the center of $B_i$, has $W^{1,\infty}$ norm at most $O(\eps^2)$; we achieve this by a simple two-piece piecewise linear interpolation in the angle variable. Recall that the total area must be $|B_i|$ to satisfy the area constraint, hence the need for an interpolation.

We start with a couple of simple lemmas regarding the area and perimeter of small perturbations of a circle. We parametrize $\S^1$ by $\gamma:[-\pi,\pi]\to \R^2$,
\[\gamma(t)=\binom{\cos t}{\sin t}
\]
and consider a normal perturbation with magnitude $u:[-\pi,\pi]\to (-1,\infty)$, which gives a variation
\[\gamma_u(t)=\gamma(t)+u(t)\nu(t)=(1+u(t))\gamma(t)
\]
where $\nu(t)=\gamma(t)$ is the outer normal. Using the formulas for the area in polar coordinates and for the length of a curve we obtain the following results, of which we omit the proof of the first.
\begin{lemma}[Variation of area]\label{lemma:taylorarea}
\begin{equation}\label{eq:taylorarea}
Area( \gamma_u)=\pi+\int_{-\pi}^\pi u(t)dt+\frac12\int_{-\pi}^\pi u(t)^2dt.
\end{equation}
\end{lemma}


\begin{lemma}[Variation of perimeter]\label{lemma:taylorperimeter}
If $u(t)\geq -\frac12$ for every $t$ and $\|u\|_{W^{1,\infty}}\leq 1$, then the length $L(\gamma_u)$ of the curve $\gamma_u$ satisfies
\begin{equation}\label{eq:taylorperimeter}
L(\gamma_u)=2\pi+\int_{-\pi}^\pi u(t)dt +\frac12 \int_{-\pi}^\pi u'(t)^2 dt+O\big(\|u\|_{W^{1,\infty}}^3\big).
\end{equation}
\end{lemma}

\begin{proof}
We have
\[\gamma_u'(t)=(1+u(t))\gamma'(t)+u'(t)\gamma(t)
\]
and by the orthogonality of $\gamma$ and $\gamma'$ we obtain
\[|\gamma_u'(t)|=\sqrt{(1+u(t))^2+u'(t)^2}=\sqrt{1+2u(t)+u(t)^2+u'(t)^2}.
\]
By the Taylor expansion with Lagrange remainder
\[\sqrt{1+x}=1+\frac12 x-\frac18 x^2+r(x)
\]
with $r(x)= \frac{1}{16(1+\xi)^{5/2}}x^3$ and $\xi$ between $1$ and $x$. Set $x=2u(t)+u(t)^2+u'(t)^2$. From $u(t)\geq-\frac12$ we obtain $x\geq -\frac34$, and then also $\xi\geq -\frac34$, thus $|r(x)|\leq C|x|^3$ for every $x\geq-\frac34$.  Therefore
\begin{align*}
L(\gamma_u)&=\int_{-\pi}^\pi |\gamma_u'(t)|dt \\
& =\int_{-\pi}^\pi \left(1+u+\frac12(u^2+u'^2)-\frac18(2u+u^2+u'^2)^2+r(u)\right)dt\\
&=2\pi+\int_{-\pi}^\pi u +\frac12 \int_{-\pi}^\pi u'^2+ O\big(\|u\|_{W^{1,\infty}}^3\big). 
\end{align*}
\end{proof}

\begin{remark} In particular consider a variation $u$ which preserves the area, that is
\[\int u=-\frac12\int u^2.
\]
Then plugging this into \eqref{eq:taylorperimeter} we obtain that for an area-preserving variation the perimeter is
\[2\pi+\frac12 \int(u'^2-u^2)+O\big(\|u\|_{W^{1,\infty}}^3\big).
\]
\end{remark}

\begin{lemma}
Consider a circle of radius $r$ centered at the origin and given $R\in\R$ consider a circle of radius $|R|$ tangent to the first one  whose center has cartesian coordinates $(r+R,0)$, (so that if $R$ is positive it is on the opposite side with respect to the tangent line, if $R$ is negative it is on the same side). Then the polar coordinates of the second circumference in a neighbourhood of the tangency point are given by:
\[\rho(\theta)=(R+r)\cos\theta-R\sqrt{1-\left(1+\frac{r}{R}\right)^2\sin^2\theta},
\]
and the Taylor expansion for small $\theta$ is
\[\rho(\theta)=r+\frac{r}{2}\left(1+\frac{r}{R}\right)\theta^2+O(\theta^4).
\]
\end{lemma}

\begin{proof}
The polar equation of a circumference of radius $R$ whose center has polar coordinates $(r_0,\phi)$ is given by
\[\rho^2+r_0^2-2\rho r_0 \cos(\theta-\phi)=R^2.
\]
In our case $(r_0,\phi)=(r+R,0)$. Inserting this into the previous equation and solving for $\rho$ (and choosing the right sign) gives the desired conclusion.
\end{proof}

\begin{theorem}[Recovery sequence]\label{thm:recovery}
For every cluster of disks $\B=(B_1,\ldots,B_N)$ with radii $r_1,\ldots,r_N$  we can construct a recovery sequence $\E_\eps$, namely a sequence such that $\E_\eps\to \B$ in the convergence of clusters and such that
\[P_\eps(\E_\eps)=\sum_{i=1}^N 2\pi r_i-\frac43\eps^{3/2}\sum_{1\leq i<j\leq N}\sigma_{ij}\frac{2r_ir_j}{r_i+r_j}+O(\eps^{5/2})
\]
where
\[\sigma_{ij}=\begin{cases}1&\text{ if $B_i$ and $B_j$ touch }\\
0& \text{ otherwise}
\end{cases}.
\]
\end{theorem}

\begin{proof}
We build, for each disk in the limit configuration, a “dented” disk, inserting small arcs of constant curvature $\kappa_{ij}=\tfrac12\big(\tfrac{1}{r_j}-\tfrac{1}{r_i}\big)$ between two tangent disks $B_i$ and $B_j$. The length of the corresponding chord is set to be $\ell_{ij}^\eps=\tfrac{4r_ir_j}{r_i+r_j}\eps^{1/2}$ (these are the asymptotically optimal values given by the optimizations in \eqref{eq:optimalkappa} and \eqref{eq:optimalell}).

We describe the boundary of $\E_\eps(i)$ in polar coordinates w.r.t. the center of $B_i$ by the function $\rho_i(\theta)$. Around any contact point $p_{ij}=(r_i,\theta_{ij})$, thanks to the previous lemma, the parametrization is given by
\[\rho_i(\theta)=r_i+\frac{r_i}{2}\left(1+\frac{r_i}{R_{ij}}\right)(\theta-\theta_{ij})^2+O((\theta-\theta_{ij})^4)
\]
where $R_{ij}=1/\kappa_{ij}$. We now suppose for simplicity $\theta_{ij}=0$ (we are interested in computing only lengths, which are rotation invariant) and compute the polar coordinates of the endpoints of the $(i,j)$-arc, whose chord has length $\ell_{ij}^\eps$: they are given by $(\rho_i(\Delta \theta_i),\Delta \theta_i)$ and $(\rho(-\Delta \theta_{ij}),-\Delta \theta_{ij})$ where $\Delta \theta_{ij}$ is implicitly given by 
\[2\rho_i(\Delta\theta_{ij})\sin\Delta \theta_{ij}=\ell_{ij}^\eps.
\]
We now invert this expression to obtain the Taylor expansion of $\Delta\theta_{ij}$ in terms of $\ell_{ij}^\eps$: first insert the Taylor expansions of $\rho_i(\theta)$ and $\sin\theta$ to obtain
\[2\left(r_i+\frac{r_i}{2}\left(1+\frac{r_i}{R_{ij}}\right)\Delta\theta_{ij}^2+O(\Delta\theta_{ij}^4)\right)\left(\Delta\theta_{ij}-\frac16 \Delta\theta_{ij}^3+O(\Delta\theta_{ij}^5\right)=\ell_{ij}^\eps.
\]
Then a simple computation yields 
\[\Delta\theta_{ij}=\frac{\ell_{ij}^\eps}{2r_i}+O(\eps^{3/2}).
\]

Using Lemma \ref{lemma:taylorperimeter} and a rescaling, setting $(1+u(\theta))r_i=\rho_i(\theta)$, and observing that we can set the total area to be $|B_i|$ with $\rho_i(\theta)$ being piecewise linear between two consecutive arcs and having there $W^{1,\infty}$-norm bounded by a constant times $\eps^2$, we find that
\begin{align*}
P(\E_\eps(i))&=r_i\left(2\pi +\frac12 \int(u'(t)^2-u(t)^2)dt+O(\eps^{5/2})\right)\\
&= r_i\left(2\pi+\frac12\int_{-\Delta\theta_{ij}}^{\Delta\theta_{ij}} u'(t)^2dt+O(\eps^{5/2})\right)\\
&=r_i\left(2\pi+\frac12\int_{-\Delta\theta_{ij}}^{\Delta\theta_{ij}} \left(1+\frac{r_i}{R_{ij}}\right)^2t^2dt+O(\eps^{5/2})\right)\\
&=2\pi r_i+\frac{r_i}{2}\frac{(r_i+R_{ij})^2}{R_{ij}^2}\frac23 \Delta\theta_{ij}^3+O(\Delta\theta_{ij}^5)\\
&=2\pi r_i+\frac13 \frac{r_iR_{ij}}{r_i+R_{ij}}\eps^{3/2}+O(\eps^{5/2})
\end{align*}
where we used that the only relevant term up to $O(\eps^{5/2})$ in the integral is $u'(t)^2$ between $-\Delta\theta_{ij}$ and $\Delta\theta_{ij}$. Moreover, recalling Lemma \ref{lemma:explicit}, we have
\[s(\ell_{ij}^\eps,\kappa_{ij})=\ell_{ij}^\eps+O(\eps^{3/2}).
\]
Therefore summing among all the arcs of the chamber $\E_\eps(i)$ we obtain
\[P(\E_\eps(i))-\frac{\eps}{2}\sum_{j} s(\ell_{ij}^\eps,\kappa_{ij})=2\pi r_i-\frac23\eps^{3/2}\sum_{j} \frac{r_iR_{ij}}{r_i+R_{ij}}+O(\eps^{5/2}).
\]
Now summing among all $i$'s, and recalling that $\frac{1}{R_{ij}}=\frac12\left(\frac{1}{r_j}-\frac{1}{r_i}\right)$, each arc $(i,j)$ is counted with a weight given by
\begin{align*}
&-\frac23\eps^{3/2}\left(\frac{1}{\frac{1}{R_{ij}}+\frac{1}{r_i}}-\frac{1}{\frac{1}{R_{ij}}-\frac{1}{r_j}}\right)\\
=& -\frac43 \eps^{3/2} \frac{2r_i r_j}{r_i+r_j}
\end{align*}
which is the desired result.
\end{proof}

\subsection{Proof of the main theorems}\label{subsec:proofs}
We now put together the previously obtained results to prove Theorem \ref{thm:sticky} and Theorem \ref{thm:structure}.

\begin{proof}[Proof of Theorem \ref{thm:sticky}]
Given a family of minimizing clusters $\overline{\E_\eps}$ converging to $\B$, by the regularity Theorem \ref{thm:regularity} they have boundary of piecewise constant curvature. By Lemma \ref{lemma:onearc} all curvatures are bounded, and every pair of chambers $\overline{\E_\eps}(i)$ and $\overline{\E_\eps}(j)$ shares at most one arc, and shares no arc if the limit disks $B_i$ and $B_j$ are not tangent. Applying Corollary \ref{cor:keyineq} we obtain that 
\begin{equation}\label{eq:lowerbound}
P_\eps(\overline{\E_\eps})\geq \sum_{i=1}^N P(B_i)-\frac43\T(\B)\eps^{3/2}+O(\eps^{5/2}),
\end{equation}
or equivalently (recalling definition \eqref{eq:P1eps}) that
\[P_\eps^{(1)}(\overline{\E_\eps})\geq -\T(\B)+O(\eps).
\]
By Theorem \ref{thm:recovery} we can actually find a recovery sequence, that is a sequence $\E_\eps$ converging to $\B$ and such that
\[P_\eps(\E_\eps)= \sum_{i=1}^N P(B_i)-\frac43\T(\B)\eps^{3/2}+O(\eps^{5/2}),
\]
which shows the other inequality in \eqref{eq:lowerbound}. In particular, 
\[P_\eps^{(1)}(\overline{\E_\eps})= -\T(\B)+O(\eps),
\]
and in order to minimize $P_\eps$ for $\eps$ small enough, it is necessary that the limit cluster $\B$ maximizes $\T(\B)$, the number of weighted tangencies.

\end{proof} 

\begin{proof}[Proof of Theorem \ref{thm:structure}]
Theorem \ref{thm:regularity} implies that there are a finite number of arcs of constant curvature, meeting in a finite number of vertices. By Lemma \ref{lemma:triplepoints} at every vertex exactly three arcs meet, one of the chambers is the exterior one and the angle $\theta_\eps$ is given by $\theta_\eps=\arccos(1-\eps/2)$. By Lemma \ref{lemma:kappabound} the curvatures of the arcs are converging to the desired values. By Lemma \ref{lemma:onearc} there is at most one arc between any pair of chambers whose limit disks are tangent, and none otherwise. Moreover, it follows from Proposition \ref{prop:keyineq} that in the former case, for $\eps$ small enough there is \emph{exactly} one arc, otherwise we would get a worse inequality from Proposition \ref{prop:keyineq}, that is $\lim_{\eps\to 0}P_\eps^{(1)}(\overline{\E_\eps})>-\T(\B)$. Finally, the length $\ell_{ij}^\eps$ of the arc between $\overline{\E_\eps}(i)$ and $\overline{\E_\eps}(j)$ must be $o(\eps^{1/2})$-close to the optimal value given by \eqref{eq:optimalell}, otherwise again we would obtain a worse inequality.

\end{proof}

\section{Final remarks}\label{sec:remarks}

\begin{enumerate}[label=(\roman*)]
\item \emph{Higher dimension.} A natural question is whether an analogous result holds for minimizing clusters in $\R^n$, where the weights are given by \ref{eq:Peps} and the length is replaced by the Hausdorff measure $\H^{n-1}$. The first-order results of Section \ref{sec:firstorder} are true in any dimension. The proof of compactness is however more subtle, as in dimension $n\geq 3$ such a strong regularity result as Theorem \ref{thm:regularity} is not available, and moreover perimeter does not control diameter even for connected smooth  sets. The localization lemma \ref{lemma:localizationsimple} is still true but requires a different proof. The second-order results of Section \ref{sec:secondorder} seem more difficult to extend, mainly because of the lack of a strong regularity result. In the planar case we are able to make explicit computations thanks to the fact that we are dealing with arcs of constant curvature.

\item \emph{The case $\eps \to 2$.} The other natural asymptotic behaviour we could consider is for $\eps\to 2$, which is the limit for the triangle inequalities \eqref{eq:wetting} to hold. In this case for minimal clusters the union of all chambers $\bigcup_{i=1}^N\overline{\E_\eps}(i)$ converges to a disk (by the isoperimetric inequality) and the cluster converges up to subsequence and rigid motions to an optimal partition of the disk. This is much simpler to prove than the main result of this paper: in this case, setting $\alpha=2-\eps$, the relevant rescaled functionals are
\[G_\alpha(\E)=\frac{P_{2-\alpha}(\E)-(1-\alpha)2\pi \sqrt{N}}{\alpha}
\]
The lower bound inequality is an immediate consequence of the rewriting
\[G_\alpha(\E)=P(\E)+\frac{1-\alpha}{\alpha}(P(\E(0))-2\pi\sqrt{N})\geq P(\E)
\]
while its sharpness (or recovery sequence in the language of $\Gamma$-convergence) is  recovered by a constant sequence.

\item \emph{Higher order expansion.} Even though Theorem \ref{thm:sticky} highly restricts the class of possible clusters of disks we can see in the limit $\eps\to 0$, it doesn't completely characterizes them because of a general lack of uniqueness of minimizers for the tangency functional $\T$ in \eqref{eq:tangency}: in the case of equal radii already for $N=6$ there are three distinct minimizers, see Figure \ref{fig:N=6}; see also \cite{DF17} for the characterization of those $N$ which admit a unique minimizer for the sticky disk potential. For those $N$ that admit many minimizers, a way to select among them would be to go beyond the order $\eps^{3/2}$ and look at the subsequent order in the expansion of perimeter. However this seems quite difficult and apparently involves some “non-local” terms. A computation in the case of equal areas seems to suggest that the relevant quantity to be maximized at the next order is the total number of paths of length $2$ in the \emph{bond graph} associated to $\B$, that is the graph where vertices are the centers of the disks and edges are drawn when two disks touch (notice that the tangency functional is exactly the number of paths of length $1$, i.e. edges, in the same graph). However there are no rigorous results in this direction.
\end{enumerate}

\begin{figure}\label{fig:N=6}
\begin{center}
\includegraphics[width=0.6\textwidth]{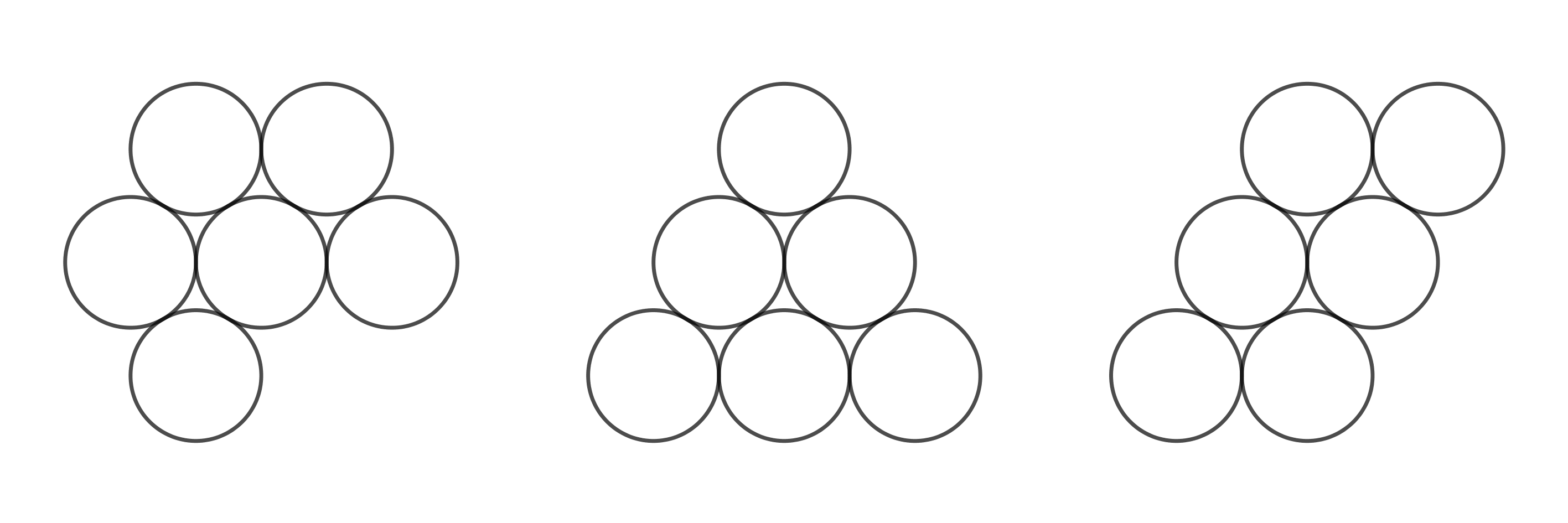}
\caption{\small{For $N=6$ and equal radii there are three distinct minimizers of the tangency functional $\T$.}}
\end{center}
\end{figure}

\section*{Acknowledgments} I want to thank Giovanni Alberti for many discussions and suggestions. I also want to thank Alessandra Pluda for the conversations from which the problem originated.


\begin{thebibliography}{99}
\bibitem{AB90II} L. Ambrosio, A. Braides: Functionals defined on partitions of sets of finite perimeter, II: semicontinuity, relaxation and homogenization.  \textit{J. Math. Pures Appl. (9)}, \textbf{69} (1990), no. 3, 307--333.

\bibitem{AnzBal93} G. Anzellotti, S. Baldo: Asymptotic development by $\Gamma$-convergence. \textit{Appl. Math. Optim.}, \textbf{27} (1993), 105--123.

\bibitem{AFS12} Y. Au Yeung, G. Friesecke, B. Schmidt: Minimizing atomic configurations of short range pair potentials in two dimensions: crystallization in the Wulff shape. \textit{Calc. Var. Partial Differential Equations}, \textbf{44} (2012), 81--100.

\bibitem{BCG17} A. Braides, S. Conti, A. Garroni: Density of polyhedral partitions. \textit{Calc. Var. Partial Differential Equations.} \textbf{56} (2017), no. 2, 10 pp.


\bibitem{Bra02} A. Braides: \textit{$\Gamma$-convergence for beginners}. Oxford University Press, Oxford, 2002.

\bibitem{CMG13} S.J. Cox, F. Morgan, F. Graner: Are large perimeter-minimizing two-dimensional clusters of equal-area bubbles hexagonal or circular?. \textit{Proc. R. Soc. Lond. Ser. A Math. Phys. Eng. Sci.}, \textbf{469} (2013), no. 2149, 10 pp. 


\bibitem{DF17} L. De Luca, G. Friesecke: Classification of particle numbers with unique Heitmann-Radin minimizer. \textit{J. Stat. Phys.}, \textbf{167} (2017), no. 6, 1586--1592. 

\bibitem{DPS16} E. Davoli, P. Piovano, U. Stefanelli: Sharp $N^{3/4}$ law for the minimizers of the edge-isoperimetric problem on the triangular lattice. \textit{J. Nonlinear Sci.}, \textbf{27} (2017), no. 2, 627--660.




\bibitem{F+93} J. Foisy, M. Alfaro, J. Brock, N. Hodges, J. Zimba: The standard double soap bubble in $\R^2$ uniquely minimizes perimeter. \textit{Pacific J. Math.}, \textbf{159} (1993), no. 1, 47--59.


\bibitem{HR80} R. Heitmann, C. Radin: Ground states for sticky disks. \textit{J. Stat. Phys.}, \textbf{22} (1980), no. 3, 281--287.


\bibitem{Law12} G. Lawlor: Double bubbles for immiscible fluids in $\R^n$. \textit{J. Geom. Anal.}, \textbf{24} (2014), no. 1, 190--204. 

\bibitem{Mag12} F. Maggi: \textit{Sets of Finite Perimeter and Geometric Variational Problems}. Cambridge University Press, Cambridge, 2012.


\bibitem{Mor98} F. Morgan: Immiscible fluid clusters in $\R^2$ and $\R^3$. \textit{Michigan Math. J.}, \textbf{45} (1998), 441--450.

\bibitem{PaoTam16} E. Paolini, A. Tamagnini: Minimal clusters of four planar regions with the same area. \textit{ESAIM:COCV}, to appear (2018).


\bibitem{Sch13} B. Schmidt: Ground states of the 2D sticky disc model: fine properties and $N^{3/4}$ law for the deviation from the asymptotic Wulff shape. \textit{J. Stat. Phys.}, \textbf{153} (2013), no. 4, 727--738. 

\bibitem{Wic04} W. Wichiramala: Proof of the planar triple bubble conjecture. \textit{J. Reine Angew. Math.}, \textbf{567} (2004), 1--49.



\end{thebibliography}
\end{document}